\documentclass{article}
\usepackage{amssymb,amsmath,amsfonts}
\usepackage{graphics}
\usepackage{graphicx}

\def\eqref#1{$(\ref{#1})$}

\newcommand{\dist}{\mathop{\mathrm{dist}}}

\newenvironment{proof}{\noindent {\em {Proof}}.}{$\square$
\medskip}
\newtheorem{theorem}{Theorem}[section]
\newtheorem{corollary}[theorem]{Corollary}
\newtheorem{lemma}[theorem]{Lemma}
\newtheorem{remark}[theorem]{Remark}
\newtheorem{proposition}[theorem]{Proposition}
\newtheorem{definition}[theorem]{Definition}
\newtheorem{example}[theorem]{Example}
\newtheorem{question}[theorem]{Question}
\newtheorem{problem}[theorem]{Problem}

\newcommand{\bd}[1]{\begin{definition}\rm\label{#1}}
\newcommand{\bt}[1]{\begin{theorem}\label{#1}}
\newcommand{\bc}[1]{\begin{corollary}\label{#1}}
\newcommand{\bl}[1]{\begin{lemma}\label{#1}}
\newcommand{\bp}[1]{\begin{proposition}\label{#1}}
\newcommand{\be}[1]{\begin{example}\rm\label{#1}}
\newcommand{\bq}[1]{\begin{question}\rm\label{#1}}
\newcommand{\bprob}[1]{\begin{problem}\rm\label{#1}}
\newcommand{\beq}[1]{\begin{eqnarray}\label{#1}}
\newcommand{\br}[1]{\begin{remark}\rm\label{#1}}

\newcommand{\el}{\end{lemma}}
\newcommand{\ep}{\end{proposition}}
\newcommand{\ee}{\end{example}}
\newcommand{\eq}{\end{question}}
\newcommand{\eprob}{\end{problem}}
\newcommand{\eeq}{\end{eqnarray}}
\newcommand{\ed}{\end{definition}}
\newcommand{\et}{\end{theorem}}
\newcommand{\ec}{\end{corollary}}
\newcommand{\er}{\end{remark}}
\title{\bf Tripartite coincidence-best proximity points in generalized metric spaces}
\author {{\sc\textbf{M. Norouzian}} and~{\sc\textbf{A. Abkar}}\footnote{corresponding author} \\
[0.3cm]
Department of Pure Mathemathics, Faculty of Science, \\
Imam Khomeini International University,
 Qazvin 34149, Iran\\
norouzian.m67@gmail.com;\, abkar@sci.ikiu.ac.ir\\
{\em}}
\date{}
\begin{document}
\maketitle \begin{abstract}
 We first introduce a notion of convex structure in generalized metric spaces, then we introduce tripartite contractions, tripartite semi-contractions, tripartite coincidence points, as well as tripartite best proximity points for a given triple $(K;S;T)$ defined on the union $A\cup B\cup C$ of closed subsets of a generalized metric space. We prove theorems on the existence and convergence of tripartite coincidence-best proximity points. 
\end{abstract}


\textbf{}{}
\section{{Introduction}}
Let $(X,d)$ be a metric space, and let $A,B$ be subsets of $X$. A mapping $T:A\cup B\rightarrow A\cup B$ is said to be \textit{cyclic}
provided that $T(A)\subseteq B$ and $T(B)\subseteq A$; similarly, a mapping $S:A\cup B\rightarrow A\cup B$ is said to be \textit{noncyclic}
if $S(A)\subseteq A$ and $S(B)\subseteq B$. The following theorem is an extension of Banach contraction principle.
\begin{theorem}(\cite{Kirk})
Let $A$ and $B$ be nonempty closed subsets of a complete metric space $(X,d)$. Suppose that $T$ is a cyclic mapping such that
\begin{equation*}
d(Tx,Ty)\leq \alpha\, d(x,y),
\end{equation*}
for some $\alpha \in (0,1)$ and for all $x\in A, ~y\in B$. Then $T$ has a unique fixed point in $A\cap B$.
\end{theorem}

Let $A$ and $B$ be nonempty subsets of a metric space $X$. A mapping $T:A\cup B\rightarrow A\cup B$ is said to be a \textit{cyclic contraction}
if $T$ is cyclic and
\begin{equation*}
d(Tx,Ty)\leq \alpha\, d(x,y)+(1-\alpha)\dist(A,B)
\end{equation*}
for some $\alpha \in (0,1)$ and for all $x\in A, ~y\in B$, where $$\dist(A,B):=\inf \{d(x,y):(x,y)\in A\times B\}.$$

For a cyclic mapping $T:A\cup B\rightarrow A\cup B$, a point $x\in A\cup B$ is
said to be a \textit{best proximity point} provided that
\begin{equation*}
d(x,Tx)=\mathrm{dist}(A,B).
\end{equation*}
The following existence, uniqueness and convergence result of a best proximity point for cyclic contractions is the main result of
\cite{Eldred}.
\begin{theorem}(\cite{Eldred})
Let $A$ and $B$ be nonempty closed convex subsets of a uniformly convex Banach space $X$ and let $T:A\cup B\rightarrow A\cup B$ be
a cyclic contraction mapping. For $x_{0}\in A$, define $x_{n+1}:=Tx_{n}$ for each $n\geq 0$. Then there exists a unique $x\in A$ such
that $x_{2n}\rightarrow x$ and
\begin{equation*}
\|x-Tx\|=\mathrm{dist}(A,B).
\end{equation*}
\end{theorem}
In the theory of best proximity points, one usually considers a cyclic mapping $T$ defined on the union of two (closed) subsets of a given metric space.
Here
the objective is to minimize the expression $d(x, Tx)$ where $x$ runs through the domain of $T$; that is $A\cup B$. In other words, we want to find
$$\mathrm{argmin}\{ d(x, Tx):  x\in A\cup B\}.$$
If $A$ and $B$ intersect, the solution is clearly a fixed point of $T$; otherwise we have
$$d(x, Tx)\ge \dist(A,B),\qquad \forall x\in A\cup B,$$
so that the point at which the equality occurs is called a best proximity point of $T$. This point of view dominates the literature.

Very recently,
N. Shahzad, M. Gabeleh, and O. Olela Otafudu \cite{Gabeleh} considered two mappings $T$ and $S$ simultaneously and established very interesting results. For
technical reasons, the first map should be cyclic and the second one should be noncyclic. According to \cite{Gabeleh}, for a nonempty pair of subsets
$(A,B)$, and a cyclic-noncyclic pair $(T;S)$ on $A\cup B$ (that is,
$T:A\cup B\rightarrow A\cup B$ is cyclic and $S:A\cup B\rightarrow A\cup B$ is noncyclic); they called a point $p\in A\cup B$
a \textit{coincidence best proximity point} for $(T;S)$ provided that
\begin{equation*}
d(Sp,Tp)=\dist(A,B).
\end{equation*}
Note that if $S=I$, the identity map on $A\cup B$, then $p\in A\cup B$ is a best proximity
point for $T$. Also, if $\dist(A,B)=0$, then $p$ is called a \textit{coincidence point} for $(T;S)$ (see \cite{Fukhar} and \cite{Garcia}
for more information). With the definition just given, and depending on the situation as to whether $S$ equals the identity map, or if the distance between
the
underlying sets is zero, one obtains a best proximity point for $T$, or a coincidence point for $T$ and $S$. This was in fact the philosophy behind the
phrase
"coincidence best proximity point" for the pair $(T;S)$.
They then defined the notion of a
cyclic-noncyclic contraction.
\begin{definition}(\cite{Gabeleh})
Let $(A,B)$ be a nonempty pair of subsets of a metric space $(X,d)$ and $T,S:A\cup B\rightarrow A\cup B$ be two mappings. The
pair $(T;S)$ is called a cyclic-noncyclic contraction pair if it satisfies the following conditions:

(1) $(T;S)$ is a cyclic-noncyclic pair on $A\cup B$.

(2) For some $r\in (0,1)$ we have
\begin{equation*}
d(Tx,Ty)\leq rd(Sx,Sy)+(1-r)\dist(A,B), ~~\forall (x,y)\in A\times B.
\end{equation*}
\end{definition}
To state the main result of \cite{Gabeleh}, we need to recall the notion of convexity in the framework of metric spaces.
In \cite{Takahashi}, Takahashi introduced the notion of convexity in metric spaces as follows (see also \cite{Shimizu}).
\begin{definition}
Let $(X,d)$ be a metric space and $I:=[0,1]$. A mapping ${\cal{W}}:X\times X\times I\rightarrow X$ is said to be a convex structure
on $X$ provided that for each $(x,y;\lambda)\in X\times X\times I$ and $u\in X$,
\begin{equation*}
d(u,{\cal{W}}(x,y;\lambda))\leq \lambda d(u,x)+(1-\lambda)d(u,y).
\end{equation*}
\end{definition}

A metric space $(X,d)$ together with a convex structure ${\cal{W}}$ is called a \textit{convex metric space} and is denoted by
$(X,d,{\cal{W}})$. A Banach space and each of its convex subsets are convex metric spaces.
A subset $K$ of a convex metric space $(X,d,{\cal{W}})$ is said to be a convex set provided that ${\cal{W}}(x,y;\lambda)\in K$
for all $x,y\in K$ and $\lambda \in I$.
Similarly,
a convex metric space $(X,d,{\cal{W}})$ is said to be uniformly convex if for any $\varepsilon >0$, there exists $\alpha=\alpha(\varepsilon)$
such that for all $r>0$ and $x,y,z\in X$ with $d(z,x)\leq r, ~d(z,y)\leq r$ and $d(x,y)\geq r\varepsilon$, we have
\begin{equation*}
d(z,{\cal{W}}(x,y;\frac{1}{2}))\leq r(1-\alpha)<r.
\end{equation*}
For example, every uniformly convex Banach space is a uniformly convex metric space.
\begin{definition}(\cite{Gabeleh})
Let $(A,B)$ be a nonempty pair of subsets of a metric space $(X,d)$. A mapping $S:A\cup B\rightarrow A\cup B$ is said to be
a relatively anti-Lipschitzian mapping if there exists $c>0$ such that
\begin{equation*}
d(x,y)\leq c\, d(Sx,Sy), ~~\forall (x,y)\in A\times B.
\end{equation*}
\end{definition}
The main result of Shahzad, et al. reads as follows:
\begin{theorem}(\cite{Gabeleh})
Let $(A,B)$ be a nonempty, closed pair of subsets of a complete uniformly convex metric space $(X,d,{\cal{W}})$ such that $A$ is
convex. Let $(T;S)$ be a cyclic-noncyclic contraction pair defined on $A\cup B$ such that $T(A)\subseteq S(B),\, T(B)\subseteq S(A)$,
and that $S$ is continuous on $A$ and relatively anti-Lipschitzian on $A\cup B$. Then $(T;S)$ has a coincidence best proximity point
in $A$. Further, if $x_{0}\in A$ and $Sx_{n+1}:=Tx_{n}$, then $(x_{2n})$ converges to the coincidence best proximity point of $(T;S)$.
\end{theorem}
Her we intend to generalize the above mentioned result in two directions. First, we consider a $G$-metric space instead of a metric space (for the
definition of $G$-metric and $G$-distance of sets, see the next section). Therefore, we have to
modify the notion of convex structure to incorporate in this new setting.
The second and more important departure point from \cite{Gabeleh} is that we instead consider a triple of mappings $(K;T;S)$ defined on the
union of three (closed) subsets of a $G$-metric space; namely $A\cup B\cup C$. We shall therefore define the new notions of coincidence point as well as
best proximity point for the triple $(K;T;S)$. Here, what we need is the concept of $G$-distance of three sets $A, B, C$, that is
\begin{equation*}
G(A,B,C):=\inf\{G(a,b,c):a\in A, b\in B, c\in C\}.
\end{equation*}
We also need to impose the right conditions on the mappings involved. This will justify the new notions of left cyclic mapping, right cyclic mapping, as
well as noncyclic mapping in the setting in which the domain has three components $A$, $B$, and $C$. This will be done in \S 3 where we define
the new concepts \textit{tripartite coincidence point} and \textit{tripartite best proximity point} for a give triple $(K;T;S)$ on $A\cup B\cup C$.
The main result of this paper is to prove existence and convergence theorems for tripartite coincidence points and tripartite best proximity points for a given triple $(K;T;S)$. In \S 3, we will introduce the new concept of tripartite contractions and will prove the mentioned results for this mappings. Finally, in \S 4, we shall introduce the notion of tripartite semi-contractions, and shall prove tripartite coincidence-best proximity points theorems for this class of mappings.

It is tempting to call these new notions as "tripled coincidence point" and "tripled best proximity point", but this phrases has already been used to indicate
particular points associated to mappings with three variables; that is for a mapping $F$ defined from $X\times X\times X$ into $X$ (see \cite{Borcut}, \cite{cho}). To avoid confusion, we have decided to adhere the adjective tripartite to this new notions. Our study is in line with the
existence of best proximity pairs which was first studied in \cite{Eldred2} by using a geometric property on a nonempty pair of subsets of a
Banach space, called \textit{proximal normal structure}, for noncyclic relatively nonexpansive mappings.

Related results on the existence and convergence of best proximity pairs can also be found in \cite{Abkar, Althagafi, Delasen, Delasen2, Dibari, Espinola, Leon, Gabeleh2, Karapinar, Norouzian1, Norouzian2, Pragadeeswarar, Suzuki} and the references therein.

\section{Convex structure in $G$-metric spaces}
In this section we first recall some necessary facts on $G$-metric spaces, a notion
introduced by
Mustafa and Sims \cite{Mustafa1} in 2006. Among other things, they characterized the
Banach fixed point theorem in $G$-metric spaces. Following their pioneering work, many authors have discussed fixed
point theorems in the framework of $G$-metric spaces; see \cite{Hussain}, \cite{Mustafa2} and \cite{VanAn}.
Second, we shall introduce a convex structure on this spaces. This convex structure will be used in the
statement of our main result in the next sections.

\begin{definition}\label{1}(\cite{Mustafa1})
Let $X$ be a nonempty set, and let $G:X\times X\times X\rightarrow \Bbb{R}^{+}$ be a function satisfying:

(1) $G(x,y,z)=0$, if $x=y=z$,

(2) $0<G(x,x,y)$, for all $x,y\in X$, with $x\neq y$,

(3) $G(x,x,y)\leq G(x,y,z)$, for all $x,y,z\in X$ with $z\neq y$,

(4) $G(x,y,z)=G(x,z,y)=G(y,z,x)=\cdots$ (symmetry in all three variables), and

(5) $G(x,y,z)\leq G(x,a,a)+G(a,y,z)$, for all $x,y,z,a\in X$, (rectangle inequality), \\
\\
then the function $G$ is called a generalized metric, or, more specifically a $G$-metric on $X$,
and the pair $(X,G)$ is called a $G$-metric space.
\end{definition}
Clearly these properties are satisfied when $G(x,y,z)$ is the perimeter of a triangle with vertices at
$x,y$ and $z$ in $\Bbb{R}^{2}$.
\begin{example}\label{3}(\cite{Mustafa1})
Let $(X,d)$ be a metric space, then
\begin{equation*}
G_{s}(d)(x,y,z)=\frac{1}{3}[d(x,y)+d(y,z)+d(x,z)]
\end{equation*}
and
\begin{equation*}
G_{m}(d)(x,y,z)=\max \{d(x,y),d(y,z),d(x,z)\}
\end{equation*}
define $G$-metrics on $X$.
\end{example}
\begin{definition}(\cite{Mustafa1})
A $G$-metric space $(X,G)$ is symmetric if
\begin{equation*}
G(x,y,y)=G(x,x,y), ~ for ~ all ~ x,y\in X.
\end{equation*}
\end{definition}
Clearly, any $G$-metric space where $G$ is derived from an underlying metric as in the above Example is symmetric.
\begin{theorem}\label{2}(\cite{Mustafa1})
Let $(X,G)$ be a $G$-metric space, then for any $x,y,z$ and $a\in X$ we have:

(1) if $G(x,y,z)=0$, then $x=y=z$,

(2) $G(x,y,z)\leq G(x,x,y)+G(x,x,z)$,

(3) $G(x,y,y)\leq 2G(y,x,x)$,

(4) $G(x,y,z)\leq G(x,a,z)+G(a,y,z)$,

(5) $G(x,y,z)\leq \frac{2}{3}[G(x,y,a)+G(x,a,z)+G(a,y,z)]$,

(6) $G(x,y,z)\leq G(x,a,a)+G(y,a,a)+G(z,a,a)$.
\end{theorem}

\begin{remark}(\cite{Mustafa1})
For any nonempty set $X$, we have seen that from any metric on $X$ we can construct a $G$-metric (see Example \ref{3}),
conversely, for any $G$-metric $G$ on $X$,
\begin{equation*}
d_{G}(x,y)=G(x,y,y)+G(x,x,y),
\end{equation*}
is readily seen to define a metric on $X$, the metric associated with $G$ satisfies
\begin{equation*}
G(x,y,z)\leq G_{s}(d_{G})(x,y,z)\leq 2G(x,y,z).
\end{equation*}
and
\begin{equation*}
\frac{1}{2}G(x,y,z)\leq G_{m}(d_{G)}(x,y,z)\leq 2G(x,y,z).
\end{equation*}
\end{remark}

\begin{theorem}(\cite{Mustafa1})
Let $(X,G)$ be $G$-metric space, then for a sequence $(x_{n})\subseteq X$ and point $x\in X$ the following are equivalent:

(1) $(x_{n})$ is $G$-convergent to $x$.

(2) $d_{G}(x_{n},x)\rightarrow 0$, as $n\rightarrow \infty$ (that is, $(x_{n})$ converges to $x$ relative to
the metric $d_{G}$).

(3) $G(x_{n},x_{n},x)\rightarrow 0$, $n\rightarrow \infty$.

(4) $G(x_{n},x,x)\rightarrow 0$, as $n\rightarrow \infty$.

(5) $G(x_{m},x_{n},x)\rightarrow 0$, as $m,n\rightarrow \infty$.
\end{theorem}
\begin{definition}(\cite{Lashkaripour})
Let $A,B$ and $C$ be three nonempty subsets of a $G$-metric space $(X,G)$. We define $G$-distance of three subsets $A,B$ and
$C$, $G(A,B,C)$, as follows:
\begin{equation*}
G(A,B,C):=\inf\{G(a,b,c):a\in A, b\in B, c\in C\}.
\end{equation*}
\end{definition}
Given $(A,B,C)$ a triple of nonempty subsets of a $G$-metric space $(X,G)$, then its proximal triple
is the triple $(A_{0},B_{0},C_{0})$ given by
\begin{equation*}
A_{0}:=\{a\in A:G(a,b,c)=G(A,B,C) ~\textrm{for some} ~b\in B,\, c\in C\},
\end{equation*}
\begin{equation*}
B_{0}:=\{b\in B:G(a,b,c)=G(A,B,C) ~\textrm{for some} ~a\in A,\, c\in C\},
\end{equation*}
\begin{equation*}
C_{0}:=\{c\in C:G(a,b,c)=G(A,B,C) ~\textrm{for some} ~a\in A,\, ~b\in B\}.
\end{equation*}
A triple of subsets $(A,B,C)$ is said to be proximal if
$A=A_{0},\, B=B_{0},$ and $C=C_{0}$.

Now it is time to introduce a convex structure on $G$-metric spaces.
We begin with the following definition.
\begin{definition}
Let $(X,G)$ be a $G$-metric space and $I:=[0,1]$. A mapping $\Bbb{W}:X\times X\times X\times I\times I\times I\rightarrow X$ is said
to be a convex structure on $(X,G)$ provided that for each $x,y,z,u,v\in X$, and each $\lambda_{1},\lambda_{2},\lambda_{3}\in I$ satisfying
$\lambda_1+\lambda_2+\lambda_3=1$, we have
\begin{equation*}
G\left (u,v,\Bbb{W}(x,y,z;\lambda_{1},\lambda_{2},\lambda_{3})\right )\le \lambda_1 G(u,v,x)+\lambda_2 G(u,v,y)+\lambda_3 G(u,v,z).
\end{equation*}
\begin{itemize}
\item
A $G$-metric space $(X,G)$ together with a convex structure $\Bbb{W}$ is called a convex $G$-metric space, and is denoted
by $(X,G,\Bbb{W})$.
\item 
A subset $U$ of a convex $G$-metric space $(X,G,\Bbb{W})$ is said to be a $G$-convex set provided that
$\Bbb{W}(x,y,z,;\lambda_{1},\lambda_{2},\lambda_{3})\in U$,
for all $x,y,z\in U$ and $\lambda_{1},\lambda_{2},\lambda_{3}\in I$ such that
$\lambda_{1}+\lambda_{2}+\lambda_{3}=1$.
\item
A convex $G$-metric space $(X,G,\Bbb{W})$ is said to be {\it uniformly convex} if for any $\varepsilon >0$, there exists
$\alpha=\alpha(\varepsilon)$ such that for all $r>0$ and $x,y,z,u,v\in X$ where $x,y,z$ are distinct points satisfying
$G(u,v,x)\leq r$, $G(u,v,y)\leq r$,
$G(u,v,z)\leq r$ and $G(x,y,z)\geq r\varepsilon$ we have
\begin{equation*}
G(u,v,\Bbb{W}(x,y,z;\frac{1}{3},\frac{1}{3},\frac{1}{3}))\leq r(1-\alpha)<r.
\end{equation*}
\end{itemize}
\end{definition}
It is clear from the definition that if $x,y,z$ are three points in $\mathbb{R}^2$, then $\Bbb{W}(x,y,z;\lambda_{1},\lambda_{2},\lambda_{3})$
is a point of the triangle with
vertices at $x,y$ and $z$ in $\Bbb{R}^{2}$. 

\begin{example}
Consider $X:=[-1,1]$ with the usual metric. For each $x,y,z\in X$ we define
\begin{equation*}
G(x,y,z):=|x-y|+|y-z|+|z-x|.
\end{equation*}
Clearly, $(X,G)$ is a $G$-metric space. Suppose that $\varepsilon>0$. For all $0<\varepsilon\leq r$ and for
all $u,v,x,y,z\in X$ in which $x,y,z$ are distinct
points and
\begin{equation*}
G(u,v,x)\leq r, ~G(u,v,y)\leq r, ~G(u,v,z)\leq r, ~G(x,y,z)\geq r\varepsilon,
\end{equation*}
we have
\begin{equation*}
G(u,v,\Bbb{W}(x,y,z;\frac{1}{3},\frac{1}{3},\frac{1}{3}))=|u-v|+|v-\frac{x+y+z}{3}|+|\frac{x+y+z}{3}-u|.
\end{equation*}
Since $x,y,z$ are distinct points and $\Bbb{W}(x,y,z;\frac{1}{3},\frac{1}{3},\frac{1}{3})$ is a point in middle of
 triangle with vertices $x,y,z$, we have
\begin{equation*}
|u-v|+|v-\frac{x+y+z}{3}|+|\frac{x+y+z}{3}-u|<\max\{|u-v|+|v-p|+|p-u|:p\in \{x,y,z\}\}.
\end{equation*}
Consequently, we have
\begin{equation*}
G(u,v,\Bbb{W}(x,y,z;\frac{1}{3},\frac{1}{3},\frac{1}{3}))<\max\{G(u,v,p):p\in \{x,y,z\}\}\leq r
\end{equation*}
which implies that
\begin{equation*}
G(u,v,\Bbb{W}(x,y,z;\frac{1}{3},\frac{1}{3},\frac{1}{3}))<r.
\end{equation*}
Thus, there exists $\beta\in (0,\frac{r}{\varepsilon})$ such that for all $u,v\in X$ and all $x,y,z\in X$ that are distinct, we have
\begin{equation*}
r-G(u,v,\Bbb{W}(x,y,z;\frac{1}{3},\frac{1}{3},\frac{1}{3}))\geq \beta \varepsilon,
\end{equation*}
and finally,
\begin{equation*}
G(u,v,\Bbb{W}(x,y,z;\frac{1}{3},\frac{1}{3},\frac{1}{3}))\leq r-\beta \varepsilon.
\end{equation*}
Now, let $(1-\alpha)r=r-\beta \varepsilon$. Therefore, by assumption that $\alpha:=\frac{\beta \varepsilon}{r}$, we have $\alpha\in (0,1)$ and
\begin{equation*}
G(u,v,\Bbb{W}(x,y,z;\frac{1}{3},\frac{1}{3},\frac{1}{3}))\leq r(1-\alpha).
\end{equation*}
Consequently, $(X,G,\Bbb{W})$ is a uniformly convex $G$-metric space.
\end{example}
\section{Tripartite coincidence-best proximity points}
We begin this section by introducing the new notions of coincidence and best proximity points.

\begin{definition}
Let $A,B$ and $C$ be nonempty subsets of a $G$-metric space $(X,G)$. Then

(1) A mapping $T:A\cup B\cup C\rightarrow A\cup B\cup C$ is said
to be a \textit{right cyclic mapping} if
\begin{equation*}
T(A)\subseteq B, ~T(B)\subseteq C, ~T(C)\subseteq A.
\end{equation*}

(2) A mapping $S:A\cup B\cup C\rightarrow A\cup B\cup C$ is said
to be a \textit{left cyclic mapping} if
\begin{equation*}
S(A)\subseteq C, ~S(C)\subseteq B, ~S(B)\subseteq A.
\end{equation*}

(3) A mapping $K:A\cup B\cup C\rightarrow A\cup B\cup C$ is said
to be a \textit{(tripartite) noncyclic mapping} if
\begin{equation*}
K(A)\subseteq A, ~K(B)\subseteq B, ~K(C)\subseteq C.
\end{equation*}
\end{definition}
\begin{example}
Let $X:=\Bbb{R}$ and
\begin{equation*}
A=\{3n\pi: ~n\in \Bbb{N}\}, ~~B=\{3n\pi +\pi: ~n\in \Bbb{N}\}, ~~C=\{3n\pi +2\pi: ~n\in \Bbb{N}\}.
\end{equation*}
Then the mappings $T,S,K:A\cup B\cup C\rightarrow A\cup B\cup C$ defined by
\begin{equation*}
Tx:=x+\pi, ~~Sx:=x+2\pi, ~~Kx:=x+3\pi,
\end{equation*}
are right cyclic, left cyclic, and (tripartite) noncyclic, respectively. \\

We will at times refer to the
triple $(T;S;K)$ as a right cyclic-left cyclic-(tripartite) noncyclic triple on $A\cup B\cup C$; or briefly as an \textbf{RLN} triple.
\end{example}
\begin{definition}
Let $A,B$ and $C$ be nonempty subsets of a $G$-metric space $(X,G)$, $T:A\cup B\cup C\rightarrow A\cup B\cup C$ be right cyclic and
$S:A\cup B\cup C\rightarrow A\cup B\cup C$ be left cyclic. A point $p\in A\cup B\cup C$ is said to be a \textsl{tripartite best proximity
point} for $T$ and $S$ provided that
\begin{equation*}
G(p,Tp,Sp)=G(A,B,C).
\end{equation*}
\end{definition}
\begin{definition}
Let $A,B$ and $C$ be nonempty subsets of a $G$-metric space $(X,G)$ and $(T;S;K)$ be an \textbf{RLN} triple
on $A\cup B\cup C$; that is, $T:A\cup B\cup C\rightarrow A\cup B\cup C$ is right cyclic,
$S:A\cup B\cup C\rightarrow A\cup B\cup C$ is left cyclic and $K:A\cup B\cup C\rightarrow A\cup B\cup C$ is tripartite
noncyclic. A point $p\in A\cup B\cup C$ is said to be a \textsl{tripartite coincidence-best proximity point} for $(T;S;K)$
provided that
\begin{equation*}
G(Kp,Tp,Sp)=G(A,B,C).
\end{equation*}
\end{definition}
Note that if in the above definition $K=I$, where $I$ denotes the identity map on $A\cup B\cup C$, then $p\in A\cup B\cup C$
will become a \textsl{tripartite best proximity point} of the mappings $T$ and $S$. Moreover, if $G(A,B,C)=0$, then $p$ will be called a
\textsl{tripartite coincidence point} for $(T;S;K)$.
\begin{definition}
Let $A,B$ and $C$ be nonempty subsets of a $G$-metric space $(X,G)$. A mapping $S:A\cup B\cup C\rightarrow A\cup B\cup C$ is
said to a left cyclic contraction if $S$ is left cyclic and
\begin{equation*}
G(Sx,Sy,Sz)\leq rG(x,y,z)
\end{equation*}
for some $r\in (0,1)$ and for all $(x,y,z)\in A\times B\times C$.
\end{definition}
\begin{definition}
Let $A,B$ and $C$ be nonempty subsets of a $G$-metric space $(X,G)$ and $T,S,K:A\cup B\cup C\rightarrow A\cup B\cup C$ be
three mappings. The triple $(T;S;K)$ is called a \textsl{tripartite contraction} if

(1) $(T;S;K)$ is an \textbf{RLN} on $A\cup B\cup C$.

(2) For some $r\in (0,1)$ and for all $(x,y,z)\in A\times B\times C$ we have
\begin{equation*}
G(Sx,Sy,Sz)\leq rG(Kx,Ky,Kz)+(1-r)G(A,B,C)
\end{equation*}
and
\begin{equation*}
G(Tx,Ty,Tz)\leq rG(Sx,Sy,Sz).
\end{equation*}
\end{definition}
\begin{example}
Let $X:=\Bbb{R}$ and let
\begin{equation*}
G(x,y,z)=|x-y|+|y-z|+|z-x|
\end{equation*}
be a $G$-metric on $X$. In addition, let
\begin{equation*}
A=\{3n\pi: ~n\in \Bbb{N}\}, ~~B=\{3n\pi +\pi: ~n\in \Bbb{N}\}, ~~C=\{3n\pi +2\pi: ~n\in \Bbb{N}\}.
\end{equation*}
Then for mappings $T,S,K:A\cup B\cup C\rightarrow A\cup B\cup C$ defined by
\begin{equation*}
Tx:=x+\pi, ~~Sx:=4x+2\pi, ~~Kx:=12x+3\pi,
\end{equation*}
for each $(x,y,z)\in A\times B\times C$, we have
\begin{flalign*}
G(Tx,Ty,Tz)&=|(x+\pi)-(y+\pi)|+|(y+\pi)-(z+\pi)|+|(z+\pi)-(x+\pi)| \\
&=|x-y|+|y-z|+|z-x|=\frac{1}{4}(|4x-4y|+|4y-4z|+|4z-4x|) \\
&=\frac{1}{4}(|(4x+2\pi)-(4y+2\pi)|+|(4y+2\pi)-(4z+2\pi)|+|(4z+2\pi)-(4x+2\pi)|) \\
&\leq \frac{1}{3}(|(4x+2\pi)-(4y+2\pi)|+|(4y+2\pi)-(4z+2\pi)|+|(4z+2\pi)-(4x+2\pi)|) \\
&=\frac{1}{3}G(Sx,Sy,Sz),
\end{flalign*}
and
\begin{flalign*}
G(Sx,Sy,Sz)&=|(4x+2\pi)-(4y+2\pi)|+|(4y+2\pi)-(4z+2\pi)|+|(4z+2\pi)-(4x+2\pi)| \\
&=|4x-4y|+|4y-4z|+|4z-4x|=\frac{1}{3}(|12x-12y|+|12y-12z|+|12z-12x|) \\
&=\frac{1}{3}(|(12x+3\pi)-(12y+3\pi)|+|(12y+3\pi)-(12z+3\pi)|+|(12z+3\pi)-(12x+3\pi)|) \\
&=\frac{1}{3}G(Kx,Ky,Kz) \\
&\leq \frac{1}{3}G(Kx,Ky,Kz)+\frac{2}{3}G(A,B,C).
\end{flalign*}
This implies that $(T;S;K)$ is a tripartite contraction with $r=\frac{1}{3}$.
\end{example}
\begin{remark}
It follows from the condition (2) of the above definition that $$G(Sx,Sy,Sz)\leq G(Kx,Ky,Kz),\quad \forall (x,y,z)\in A\times B\times C.$$ Moreover,
if $K$ is a noncyclic relatively nonexpansive mapping; meaning that $$G(Kx,Ky,Kz)\leq G(x,y,z)\quad \forall (x,y,z)\in A\times B\times C,$$
then $S$ is a left cyclic contraction. In addition, if in the above definition $K$ is $G$-continuous, then $S$ and $T$ are $G$-continuous as well
(that is, they are continuous with respect to the topology induced by the $G$-metric).
\end{remark}
\begin{lemma}\label{4}
Let $(A,B,C)$ be a triple of nonempty subsets of a $G$-metric space $(X,G)$ and let $(T;S;K)$ be a tripartite contraction on
$A\cup B\cup C$. Suppose that $T(A)\subseteq S(C)\subseteq K(B)$, $T(B)\subseteq S(A)\subseteq K(C)$ and
$T(C)\subseteq S(B)\subseteq K(A)$. Then there exists a sequence $(x_{n})$ in $X$ such that $Tx_{n}=Sx_{n+1}=Kx_{n+2}$
for each $n\geq 0$; moreover, $(x_{3n}), ~(x_{3n+1})$ and $(x_{3n+2})$ are sequences in $A,C$ and $B$ respectively, and
\begin{equation*}
G(Tx_{3n},Tx_{3n+1},Tx_{3n+2})\rightarrow 0,\quad n\to\infty.
\end{equation*}
\end{lemma}
\begin{proof}
Let $x_{0}\in A$. Since $Tx_{0}\in S(C)\subseteq K(B)$ and $T(C)\subseteq S(B)$, there exist $x_{1}\in C, ~x_{2}\in B$ such that
$Tx_{0}=Sx_{1}=Kx_{2}$ and $Tx_{1}=Sx_{2}\in K(A)$. We know that $T(B)\subseteq S(A)$, therefore there exists $x_{3}\in A$
such that $Tx_{1}=Sx_{2}=Kx_{3}$ and $Tx_{2}=Sx_{3}\in K(C)$.

Again, since $T(A)\subseteq S(C)$, there exists $x_{4}\in C$ such that $Tx_{2}=Sx_{3}=Kx_{4}$ and $Tx_{3}=Sx_{4}\in K(B)$.
Since $T(C)\subseteq S(B)$, there exists $x_{5}\in B$ such that $Tx_{3}=Sx_{4}=Kx_{5}$ and $Tx_{4}=Sx_{5}\in K(A)$.

Continuing this process, we obtain a sequence $(x_{n})$ such that
$(x_{3n}), ~(x_{3n+1})$ and $(x_{3n+2})$ are sequences in $A,C$ and $B$, respectively, and $Tx_{n}=Sx_{n+1}=Kx_{n+2}$
for each $n\geq 0$. Since $(T;S;K)$ is a tripartite contraction, we have
\begin{flalign*}
G(Tx_{3n},Tx_{3n+1},Tx_{3n+2})&\leq rG(Sx_{3n},Sx_{3n+1},Sx_{3n+2})=rG(Tx_{3n-1},Tx_{3n},Tx_{3n+1}) \\
&\leq r[rG(Sx_{3n-1},Sx_{3n},Sx_{3n+1})]=r^{2}G(Tx_{3n-2},Tx_{3n-1},Tx_{3n}) \\
&\leq \cdots \leq r^{3n}G(Tx_{0},Tx_{1},Tx_{2}).
\end{flalign*}
Letting $n\rightarrow \infty$, we obtain
\begin{equation*}
G(Tx_{3n},Tx_{3n+1},Tx_{3n+2})\rightarrow 0.
\end{equation*}
\end{proof}

Note that in the above lemma, for each $n\geq 1$, $(x_{3n-1})\subseteq B$ and for each $n\geq 2$, $(x_{3n-2})\subseteq C$.
\begin{lemma}\label{5}
Let $(A,B,C)$ be a nonempty triple of subsets of a $G$-metric space $(X,G)$ and let $(T;S;K)$ be a tripartite contraction on
$A\cup B\cup C$. Suppose that $T(A)\subseteq S(C)\subseteq K(B)$, $T(B)\subseteq S(A)\subseteq K(C)$ and
$T(C)\subseteq S(B)\subseteq K(A)$. For $x_{0}\in A$, define $Tx_{n}=Sx_{n+1}=Kx_{n+2}$
for each $n\geq 0$. Then we have
\begin{equation*}
G(Kx_{3n},Kx_{3n+1},Kx_{3n+2})\rightarrow G(A,B,C).
\end{equation*}
\end{lemma}
\begin{proof}
Since $(T;S;K)$ is a tripartite contraction, we have
\begin{flalign*}
G(Kx_{3n},Kx_{3n+1},Kx_{3n+2})&=G(Sx_{3n-1},Sx_{3n},Sx_{3n+1}) \\
&\leq rG(Kx_{3n-1},Kx_{3n},Kx_{3n+1})+(1-r)G(A,B,C) \\
&=rG(Sx_{3n-2},Sx_{3n-1},Sx_{3n})+(1-r)G(A,B,C) \\
&\leq r[rG(Kx_{3n-2},Kx_{3n-1},Kx_{3n})+(1-r)G(A,B,C)]+(1-r)G(A,B,C) \\
&=r^{2}G(Sx_{3n-3},Sx_{3n-2},Sx_{3n-1})+(1-r^{2})G(A,B,C) \\
&\leq \cdots \leq r^{3n}G(Kx_{0},Kx_{1},Kx_{2})+(1-r^{3n})G(A,B,C).
\end{flalign*}
Letting $n\rightarrow \infty$, we obtain
\begin{equation*}
G(Kx_{3n},Kx_{3n+1},Kx_{3n+2})\rightarrow G(A,B,C).
\end{equation*}
\end{proof}

In the following, we shall establish a theorem on the existence of tripartite coincidence-best proximity point.
\begin{theorem}\label{6}
Let $(A,B,C)$ be a nonempty triple of subsets of a $G$-metric space $(X,G)$ and let $(T;S;K)$ be a tripartite contraction on
$A\cup B\cup C$. Suppose that $T(A)\subseteq S(C)\subseteq K(B)$, $T(B)\subseteq S(A)\subseteq K(C)$ and
$T(C)\subseteq S(B)\subseteq K(A)$ and $K$ is $G$-continuous on $A\cup B\cup C$. For $x_{0}\in A$, define $Tx_{n}=Sx_{n+1}=Kx_{n+2}$ for each $n\geq 0$.
If $(x_{3n})$ has a $G$-convergent subsequence in $A$, then the triple $(T;S;K)$ has a tripartite coincidence-best proximity
point in $A$.
\end{theorem}
\begin{proof}
Let $(x_{3n_{k}})$ be a subsequence of $(x_{3n})$ such that $x_{3n_{k}}\rightarrow p\in A$ (in the topology induced by the $G$-metric).
By Lemmas \ref{4} and \ref{5}, if
$k\rightarrow \infty$, we obtain
\begin{equation*}
G(Tp,Tx_{3n_{k}-1},Tx_{3n_{k}-2})\rightarrow 0
\end{equation*}
and
\begin{equation*}
G(Kp,Kx_{3n_{k}+1},Kx_{3n_{k}+2})\rightarrow G(A,B,C).
\end{equation*}
It follows from the continuity of $K$ and Theorem \ref{2} that
\begin{align*}
G(A,B,C)&\leq G(Kp,Tp,Sp)\leq G(Kp,Sx_{3n_{k}},Sx_{3n_{k}})+G(Tp,Sx_{3n_{k}},Sx_{3n_{k}})+G(Sp,Sx_{3n_{k}},Sx_{3n_{k}}) \\
&\leq G(Kp,Kx_{3n_{k}+1},Kx_{3n_{k}+1})+G(Tp,Tx_{3n_{k}-1},Tx_{3n_{k}-1})+G(Sp,Sx_{3n_{k}},Sx_{3n_{k}}) \\
&\leq G(Kp,Kx_{3n_{k}+1},Kx_{3n_{k}+2})+G(Tp,Tx_{3n_{k}-1},Tx_{3n_{k}-2})+G(Sp,Sx_{3n_{k}},Sx_{3n_{k}}).
\end{align*}
Letting $k\rightarrow \infty$, we conclude that
\begin{equation*}
G(Kp,Tp,Sp)=G(A,B,C).
\end{equation*}
\end{proof}

To obtain the second result on the existence of tripartite coincidence-best proximity points, we need some preparations.
Let $(X,G)$ be a $G$-metric space, then a sequence $(x_{n})\subseteq X$ is said to be $G$-bounded if there exists
$x_{0},y_{0}\in X$ and $M>0$ such that for each $n\in \Bbb{N}$ we have
\begin{equation*}
G(x_{n},x_{0},y_{0})\leq M.
\end{equation*}

Note that if $(x_{n})$ be a $G$-bounded sequence, then for all $n,m,l\in \Bbb{N}$ we have
\begin{flalign*}
G(x_{n},x_{m},x_{l})&\leq G(x_{n},x_{0},x_{0})+G(x_{m},x_{0},x_{0})+G(x_{l},x_{0},x_{0}) \\
&\leq G(x_{n},x_{0},y_{0})+G(x_{m},x_{0},y_{0})+G(x_{l},x_{0},y_{0}) \\
&\leq M+M+M=3M.
\end{flalign*}
This implies that
\begin{equation*}
\sup \{G(x_{n},x_{m},x_{l}): ~n,m,l\in \Bbb{N}\}<\infty.
\end{equation*}
\begin{lemma}\label{7}
Let $(A,B,C)$ be a nonempty triple of subsets of a $G$-metric space $(X,G)$ and let $(T;S;K)$ be a tripartite contraction on
$A\cup B\cup C$. Suppose that $T(A)\subseteq S(C)\subseteq K(B)$, $T(B)\subseteq S(A)\subseteq K(C)$ and
$T(C)\subseteq S(B)\subseteq K(A)$, moreover $S$ and $K$ commute on $A\cup C$. For $x_{0}\in A$, define $Tx_{n}=Sx_{n+1}=Kx_{n+2}$
for each $n\geq 0$. Then $(Kx_{3n})$, $(Kx_{3n+1})$ and $(Kx_{3n+2})$ are $G$-bounded sequences in $A,\, C$ and $B$, respectively.
\end{lemma}
\begin{proof}
Since $G(Kx_{3n},Kx_{3n+1},Kx_{3n+2})\rightarrow G(A,B,C)$, it suffices to show that $(Kx_{3n})$ is $G$-bounded in $A$.
Suppose to the contrary that there exists $N_{0}\in \Bbb{N}$ such that
\begin{equation*}
G(S(Kx_{1}),S(Kx_{2}),Kx_{3_{N_{0}+1}})> M, ~~G(S(Kx_{1}),S(Kx_{2}),Kx_{3_{N_{0}-1}})\leq M,
\end{equation*}
where,
\begin{equation*}
M> \max \left \{\frac{r^{2}}{1-r^{2}}G(K(Kx_{0}),K(Kx_{1}),S(Kx_{1}))+G(A,B,C), ~~G(Sx_{2},Sx_{0},S(Kx_{1}))\right \}.
\end{equation*}
Hence, we have
\begin{flalign*}
\frac{M-G(A,B,C)}{r^{2}}&+G(A,B,C)< \frac{G(S(Kx_{1}),S(Kx_{2}),Kx_{3_{N_{0}+1}})-G(A,B,C)}{r^{2}}+G(A,B,C) \\
&\leq \frac{G(S(Kx_{1}),S(Kx_{2}),Kx_{3_{N_{0}+1}})+(r^{2}-1)G(A,B,C)}{r^{2}} \\
&\leq \frac{G(S(Kx_{1}),S(Kx_{2}),Kx_{3_{N_{0}+1}})+(r^{2}-1)G(S(Kx_{1}),S(Kx_{2}),Kx_{3_{N_{0}+1}})}{r^{2}} \\
&=G(S(Kx_{1}),S(Kx_{2}),Kx_{3_{N_{0}+1}})=G(S(Kx_{1}),S(Kx_{2}),Sx_{3_{N_{0}}}) \\
&\leq G(K(Kx_{1}),K(Kx_{2}),Kx_{3_{N_{0}}})=G(K(Sx_{0}),K(Sx_{1}),Sx_{3_{N_{0}-1}}) \\
&=G(S(Kx_{0}),S(Kx_{1}),Sx_{3_{N_{0}-1}})\leq G(K(Kx_{0}),K(Kx_{1}),Kx_{3_{N_{0}-1}}) \\
&\leq G(K(Kx_{0}),K(Kx_{1}),S(Kx_{1})+G(S(Kx_{1}),S(Kx_{1}),Kx_{3_{N_{0}-1}}) \\
&\leq G(K(Kx_{0}),K(Kx_{1}),S(Kx_{1})+G(S(Kx_{1}),S(Kx_{2}),Kx_{3_{N_{0}-1}}) \\
&\leq G(K(Kx_{0}),K(Kx_{1}),S(Kx_{1})+M
\end{flalign*}
Thus
\begin{equation*}
\frac{M-G(A,B,C)}{r^{2}}+G(A,B,C)< G(K(Kx_{0}),K(Kx_{1}),S(Kx_{1}))+M,
\end{equation*}
and so,
\begin{equation*}
M-(1-r^{2})G(A,B,C)< r^{2}\{G(K(Kx_{0}),K(Kx_{1}),S(Kx_{1}))+M\}.
\end{equation*}
This implies that
\begin{equation*}
M< \frac{r^{2}}{1-r^{2}}G(K(Kx_{0}),K(Kx_{1}),S(Kx_{1}))+G(A,B,C),
\end{equation*}
which is a contradiction with the choice of $M$.
\end{proof}
\begin{definition}
Let $(A,B,C)$ be a nonempty triple of subsets of a $G$-metric space $(X,G)$. A mapping $K:A\cup B\cup C \rightarrow
A\cup B\cup C \rightarrow$ is said to be a \textsl{tripartite relatively anti-Lipschitzian mapping} if there exists $c>0$ such
that for all $(x,y,z)\in A\times B\times C$ we have
\begin{equation*}
G(x,y,z)\leq c\,G(Kx,Ky,Kz).
\end{equation*}
\end{definition}
The next theorem is a straightforward consequence of Theorem \ref{6} and Lemma \ref{7}. We just recall that a subset $A$ in a $G$-metric space $(X,G)$
is said to be $G$-compact if every $G$-bounded sequence in $A$ has a $G$-convergent subsequence in $A$.
\begin{theorem}\label{11}
Let $(A,B,C)$ be a nonempty triple of subsets of a $G$-metric space $(X,G)$ such that $A$ is $G$-sequentially compact, and let
$(T;S;K)$ be a tripartite contraction on $A\cup B\cup C$. Suppose that $T(A)\subseteq S(C)\subseteq K(B)$,
$T(B)\subseteq S(A)\subseteq K(C)$ and $T(C)\subseteq S(B)\subseteq K(A)$, moreover $S$ and $K$ commute on $A\cup C$. If $K$ is
tripartite relatively anti-Lipschitzian and $G$-continuous on $A\cup B\cup C$, then the triple $(T;S;K)$ has a
tripartite coincidence-best proximity point in $A$.
\end{theorem}

We now will illustrate Theorem \ref{11} with the following examples.
\begin{example}
Consider $X:=[0,+\infty)$ with the usual metric. We have already seen that the function $G:X\times X\times X\rightarrow [0,+\infty)$ defined by
\begin{equation*}
G(x,y,z)=|x-y|+|y-z|+|z-x|,
\end{equation*}
for all $x,y,z\in X$, is a $G$-metric on $X$. For $A=[0,1], B=[0,2]$ and $C=[0,3]$ we have $G(A,B,C)=0.$ We define
\begin{equation*}
Tx:=\frac{1}{4}\sin x, ~~Sx:=\frac{1}{2}\sin x, ~~Kx:=x, ~~\forall x\in A\cup B\cup C.
\end{equation*}
Then for each $(x,y,z)\in A\times B\times C$, we have
\begin{flalign*}
G(Tx,Ty,Tz)&=|\frac{1}{4}\sin x-\frac{1}{4}\sin y|+|\frac{1}{4}\sin y-\frac{1}{4}\sin z|+|\frac{1}{4}\sin z-\frac{1}{4}\sin x| \\
&=\frac{1}{4}(|\sin x-\sin y|+|\sin y-\sin z|+|\sin z-\sin x|) \\
&\leq \frac{3}{4}[\frac{1}{2}(|\sin x-\sin y|+|\sin y-\sin z|+|\sin z-\sin x|)] \\
&=\frac{3}{4}G(Sx,Sy,Sz).
\end{flalign*}
In addition, for each $(x,y,z)\in A\times B\times C$, we obtain
\begin{flalign*}
G(Sx,Sy,Sz)&=|\frac{1}{2}\sin x-\frac{1}{2}\sin y|+|\frac{1}{2}\sin y-\frac{1}{2}\sin z|+|\frac{1}{2}\sin z-\frac{1}{2}\sin x| \\
&=\frac{1}{2}(|\sin x-\sin y|+|\sin y-\sin z|+|\sin z-\sin x|) \\
&\leq \frac{3}{4}(|x-y|+|y-z|+|z-x|)+\frac{1}{4}G(A,B,C) \\
&=\frac{3}{4}G(Kx,Ky,Kz)+\frac{1}{4}G(A,B,C).
\end{flalign*}
This implies that $(T;S;K)$ is a tripartite contraction with $r=\frac{3}{4}$. Also, $T(A)\subseteq S(C)\subseteq K(B)$,
$T(B)\subseteq S(A)\subseteq K(C)$ and $T(C)\subseteq S(B)\subseteq K(A)$. Moreover, $K$ is $G$-continuous on
$A\cup B\cup C$ and $A$ is $G$-sequentially compact in $X$.
Besides, $K$ is tripartite relatively anti-Lipschitzian on $A\cup B\cup C$ with $c=2$. In fact, for all $(x,y,z)\in A\times B\times C$
we have
\begin{equation*}
G(x,y,z)=|x-y|+|y-z|+|z-x|\leq 2(|x-y|+|y-z|+|z-x|)=2G(Kx,Ky,Kz).
\end{equation*}
Finally, for each $x\in A\cup C$ we have
\begin{equation*}
K(Sx)=K(\frac {1}{2}\sin x)=\frac{1}{2}\sin x=Sx=S(Kx),
\end{equation*}
that is, $S$ and $K$ commute on $A \cup C$. Thereby, the existence of tripartite coincidence-best proximity point for
$(T;S;K)$ follows from Theorem \ref{11}. That is, there exists $p\in A$ such that
\begin{equation*}
G(Kp,Tp,Sp)=G(A,B,C)
\end{equation*}
or
\begin{equation*}
G(p,\frac{1}{4}\sin p,\frac{1}{2}\sin p)=0.
\end{equation*}
Therefore,
\begin{equation*}
p=\frac{1}{4}\sin p=\frac{1}{2}\sin p,
\end{equation*}
which implies that $p=0$.
\end{example}
\begin{example}
Let $X:=[0,+\infty)$ with the usual metric. Again $G:X\times X\times X\rightarrow [0,+\infty)$ defined by
\begin{equation*}
G(x,y,z)=|x-y|+|y-z|+|z-x|,
\end{equation*}
for all $x,y,z\in X$, is a $G$-metric on $X$. For $A=[0,1], B=[0,2]$ and $C=[0,3]$ we have $G(A,B,C)=0$; now we define
\begin{equation*}
Tx:=\frac{1}{3}\sin x, ~~Sx:=\frac{1}{2}x, ~~Kx:=x, ~~\forall x\in A\cup B\cup C.
\end{equation*}
Then for each $(x,y,z)\in A\times B\times C$, we have
\begin{flalign*}
G(Tx,Ty,Tz)&=|\frac{1}{3}\sin x-\frac{1}{3}\sin y|+|\frac{1}{3}\sin y-\frac{1}{3}\sin z|+|\frac{1}{3}\sin z-\frac{1}{3}\sin x| \\
&=\frac{1}{3}(|\sin x-\sin y|+|\sin y-\sin z|+|\sin z-\sin x|) \\
&\leq \frac{1}{3}(|x-y|+|y-z|+|z-x|)=\frac{2}{3}[\frac{1}{2}(|x-y|+|y-z|+|z-x|)] \\
&=\frac{2}{3}G(Sx,Sy,Sz)\leq \frac{5}{6}G(Sx,Sy,Sz).
\end{flalign*}
In addition, for each $(x,y,z)\in A\times B\times C$, we have
\begin{flalign*}
G(Sx,Sy,Sz)&=|\frac{1}{2}x-\frac{1}{2}y|+|\frac{1}{2}y-\frac{1}{2}z|+|\frac{1}{2}z-\frac{1}{2}x| \\
&=\frac{1}{2}(|x-y|+|y-z|+|z-x|)=\frac{1}{2}G(Kx,Ky,Kz) \\
&\leq \frac{5}{6}G(Kx,Ky,Kz)\leq \frac{5}{6}G(Kx,Ky,Kz)+\frac{1}{6}G(A,B,C).
\end{flalign*}
This implies that $(T;S;K)$ is a tripartite contraction with $r=\frac{5}{6}$. Also, $T(A)\subseteq S(C)\subseteq K(B)$,
$T(B)\subseteq S(A)\subseteq K(C)$ and $T(C)\subseteq S(B)\subseteq K(A)$. Moreover, $K$ is $G$-continuous on
$A\cup B\cup C$ and $A$ is $G$-sequentially compact in $X$.
Besides, $K$ is tripartite relatively anti-Lipschitzian on $A\cup B\cup C$ with $c=2$. In fact, for all $(x,y,z)\in A\times B\times C$
we have
\begin{equation*}
G(x,y,z)=|x-y|+|y-z|+|z-x|\leq 2(|x-y|+|y-z|+|z-x|)=2G(Kx,Ky,Kz).
\end{equation*}
Finally, for each $x\in A\cup C$ we have
\begin{equation*}
K(Sx)=K(\frac {1}{2}x)=\frac{1}{2}x=Sx=S(Kx),
\end{equation*}
that is, $S$ and $K$ commute on $A \cup C$. Thereby, the existence of tripartite coincidence-best proximity point for
$(T;S;K)$ follows from Theorem \ref{11}. That is, there exists $p\in A$ such that
\begin{equation*}
G(Kp,Tp,Sp)=G(A,B,C)
\end{equation*}
or
\begin{equation*}
G(p,\frac{1}{3}\sin p,\frac{1}{2}p)=0.
\end{equation*}
Therefore,
\begin{equation*}
p=\frac{1}{3}\sin p=\frac{1}{2}p,
\end{equation*}
which implies that $p=0$.
\end{example}
So far, we have been dealing with the existence of tripartite coincidence-best proximity points of tripartite contractions. Now we want to approximate
these points. To achieve this goal, we need the convex structure of $G$-metric space.
We begin with the following lemma.
\begin{lemma}\label{8}
Let $(A,B,C)$ be a nonempty triple of subsets of a uniformly convex $G$-metric space $(X,G;\Bbb{W})$ such that $A$ is $G$-convex.
Let $(T;S;K)$ be a tripartite contraction on $A\cup B\cup C$ such that $T(A)\subseteq S(C)\subseteq K(B)$,
$T(B)\subseteq S(A)\subseteq K(C)$ and $T(C)\subseteq S(B)\subseteq K(A)$. For $x_{0}\in A$, define $Tx_{n}=Sx_{n+1}=Kx_{n+2}$
for each $n\geq 0$. Then
\begin{align*}
G(Kx_{3n},Kx_{3n+3},Kx_{3n+6})\rightarrow 0,
\end{align*}
\begin{align*}
G(Kx_{3n+1},Kx_{3n+4},Kx_{3n+7})\rightarrow 0,
\end{align*}
and
\begin{equation*}
G(Kx_{3n+2},Kx_{3n+5},Kx_{3n+8})\rightarrow 0.
\end{equation*}
\end{lemma}
\begin{proof}
We want to show that $G(Kx_{3n},Kx_{3n+3},Kx_{3n+6})\rightarrow 0$ as $n\to \infty$. Suppose to the contrary that there exists $\varepsilon_{0}> 0$ such
that for each $k\geq 1$, there exists $n_{k}\geq k$ such that
$$G(Kx_{3n_{k}},Kx_{3n_{k}+3},Kx_{3n_{k}+6})\geq \varepsilon_{0}.$$
Choose
$0<\gamma <1$ such that $\frac{\varepsilon_{0}}{\gamma}>G(A,B,C)$ and choose $\varepsilon >0$ such that
\begin{equation*}
0< \min \left\{\frac{\varepsilon_{0}}{\gamma}-G(A,B,C), ~\frac{G(A,B,C)\alpha(\gamma)}{1-\alpha(\gamma)}\right\}.
\end{equation*}
By Lemma \ref{5}, since $G(Kx_{3n_{k}},Kx_{3n_{k}+1},Kx_{3n_{k}+2})\rightarrow G(A,B,C)$, there exists $N\in \Bbb{N}$ such that
\begin{align*}
G(Kx_{3n_{k}+2},Kx_{3n_{k}+1},Kx_{3n_{k}})\leq G(A,B,C)+\varepsilon, \\
G(Kx_{3n_{k}+2},Kx_{3n_{k}+1},Kx_{3n_{k}+3})\leq G(A,B,C)+\varepsilon,
\end{align*}
\begin{equation*}
G(Kx_{3n_{k}+2},Kx_{3n_{k}+1},Kx_{3n_{k}+6})\leq G(A,B,C)+\varepsilon,
\end{equation*}
and
\begin{equation*}
G(Kx_{3n_{k}},Kx_{3n_{k}+3},Kx_{3n_{k}+6})\geq \varepsilon_{0}> \gamma(G(A,B,C)+\varepsilon).
\end{equation*}
It now follows from the uniform convexity of $(X,G)$ and the $G$-convexity of $A$ that
\begin{flalign*}
G(A,B,C)&\leq G(Kx_{3n_{k}+2},Kx_{3n_{k}+1},\Bbb{W}(Kx_{3n_{k}},Kx_{3n_{k}+3},Kx_{3n_{k}+6},\frac{1}{3},\frac{1}{3},\frac{1}{3})) \\
&\leq (G(A,B,C)+\varepsilon)(1-\alpha(\gamma)) \\
&< G(A,B,C)+\frac{G(A,B,C)\alpha(\gamma)}{1-\alpha(\gamma)}(1-\alpha(\gamma))=G(A,B,C).
\end{flalign*}
which is a contradiction. Similarly, we see that
\begin{align*}
G(Kx_{3n+1},Kx_{3n+4},Kx_{3n+7})\rightarrow 0, \\ G(Kx_{3n+2},Kx_{3n+5},Kx_{3n+8})\rightarrow 0.
\end{align*}
This completes the proof.
\end{proof}
\begin{theorem}\label{9}
Let $(A,B,C)$ be a triple of nonempty, closed subsets of a complete uniformly convex $G$-metric space $(X,G,\Bbb{W})$
such that $A$ is $G$-convex. Let $(T;S;K)$ be a tripartite contraction on $A\cup B\cup C$ such that
 $T(A)\subseteq S(C)\subseteq K(B)$, $T(B)\subseteq S(A)\subseteq K(C)$ and $T(C)\subseteq S(B)\subseteq K(A)$ and that
 $K$ is $G$-continuous and tripartite relatively anti-Lipschitzian on $A\cup B\cup C$. Then
 $(T;S;K)$ has a tripartite coincidence-best proximity point in $A$. Moreover, if $x_{0}\in A$ and $Tx_{n}=Sx_{n+1}=Kx_{n+2}$,
 then $(x_{3n})$ $G$-converges to the tripartite coincidence-best proximity point of $(T;S;K)$.
\end{theorem}
\begin{proof}
For $x_{0}\in A$ define $Tx_{n}=Sx_{n+1}=Kx_{n+2}$ for each $n\geq 0$. We prove that $(Kx_{3n})$, $(Kx_{3n+1})$ and $(Kx_{3n+2})$
are $G$-Cauchy sequences. At first, we verify that for each $\varepsilon >0$ there exists $N_{0}\in \Bbb{N}$ such that
\begin{equation*}
G(Kx_{3m},Kx_{3n+1},Kx_{3n+2})< G(A,B,C)+\varepsilon, ~~\forall m>n\geq N_{0}. ~~~~~(*)
\end{equation*}
Assume the contrary. Then there exists $\varepsilon_{0}>0$ such that for each $k\geq 1$ there exists $m_{k}>n_{k}\geq k$
satisfying
\begin{align*}
G(Kx_{3m_{k}},Kx_{3n_{k}+1},Kx_{3n_{k}+2})\geq G(A,B,C)+\varepsilon_{0},\end{align*}
and
\begin{align*}
G(Kx_{3m_{k}-3},Kx_{3n_{k}+1},Kx_{3n_{k}+2})< G(A,B,C)+\varepsilon_{0}.
\end{align*}
Now, we have
\begin{flalign*}
G(A,B,C)+\varepsilon_{0}&\leq G(Kx_{3m_{k}},Kx_{3n_{k}+1},Kx_{3n_{k}+2})=G(Sx_{3m_{k}-1},Sx_{3n_{k}},Sx_{3n_{k}+1}) \\
&\leq G(Kx_{3m_{k}-1},Kx_{3n_{k}},Kx_{3n_{k}+1})=G(Sx_{3m_{k}-2},Sx_{3n_{k}-1},Sx_{3n_{k}}) \\
&\leq G(Kx_{3m_{k}-2},Kx_{3n_{k}-1},Kx_{3n_{k}})=G(Sx_{3m_{k}-3},Sx_{3n_{k}-2},Sx_{3n_{k}-1}) \\
&\leq G(Kx_{3m_{k}-3},Kx_{3n_{k}-2},Kx_{3n_{k}-1}) \\
&\leq G(Kx_{3m_{k}-3},Kx_{3n_{k}+1},Kx_{3n_{k}+1})+G(Kx_{3n_{k}-2},Kx_{3n_{k}+1},Kx_{3n_{k}+1}) \\
&+G(Kx_{3n_{k}-1},Kx_{3n_{k}+1},Kx_{3n_{k}+1}) \\
&\leq G(Kx_{3m_{k}-3},Kx_{3n_{k}+1},Kx_{3n_{k}+2})+G(Kx_{3n_{k}-2},Kx_{3n_{k}+1},Kx_{3n_{k}+4}) \\
&+ G(Kx_{3n_{k}-1},Kx_{3n_{k}},Kx_{3n_{k}+1}) \\
&= G(Kx_{3m_{k}-3},Kx_{3n_{k}+1},Kx_{3n_{k}+2})+G(Kx_{3n_{k}-2},Kx_{3n_{k}+1},Kx_{3n_{k}+4}) \\
&+ G(Tx_{3n_{k}-3},Tx_{3n_{k}-2},Tx_{3n_{k}-1}).
\end{flalign*}
Letting $k\rightarrow \infty$, and using the hypothesis together with Lemmas \ref{4} and \ref{8} we obtain
\begin{equation*}
G(Kx_{3m_{k}},Kx_{3n_{k}+1},Kx_{3n_{k}+2})\rightarrow G(A,B,C)+\varepsilon_{0}.
\end{equation*}
Besides,
\begin{flalign*}
G(A,B,C)+\varepsilon_{0}&\leq G(Kx_{3m_{k}},Kx_{3n_{k}+1},Kx_{3n_{k}+2})=G(Tx_{3m_{k}-2},Tx_{3n_{k}-1},Tx_{3n_{k}}) \\
&\leq rG(Sx_{3m_{k}-2},Sx_{3n_{k}-1},Sx_{3n_{k}})\leq rG(Kx_{3m_{k}-2},Kx_{3n_{k}-1},Kx_{3n_{k}}) \\
&=rG(Sx_{3m_{k}-3},Sx_{3n_{k}-2},Sx_{3n_{k}-1})\leq rG(Kx_{3m_{k}-3},Kx_{3n_{k}-2},Kx_{3n_{k}-1}) \\
&\leq rG(Kx_{3m_{k}-3},Kx_{3n_{k}-2},Kx_{3n_{k}-1})+(1-r)G(A,B,C).
\end{flalign*}
Letting $k\rightarrow \infty$, we conclude that
\begin{equation*}
G(A,B,C)+\varepsilon_{0}\leq r(G(A,B,C)+\varepsilon_{0})+(1-r)G(A,B,C)\leq G(A,B,C)+\varepsilon_{0}.
\end{equation*}
This implies that $r=1$, which is a contradiction. That is, $(*)$ holds. Similarly, we see that
\begin{equation*}
G(Kx_{3l_{k}},Kx_{3l_{k}+1},Kx_{3l_{k}+2})< G(A,B,C)+\varepsilon_{0}.
\end{equation*}
Now, suppose $(Kx_{3n})$ is not a $G$-Cauchy sequence. Then there exists $\varepsilon_{0}> 0$ such that for each $k\geq 1$ there exist
$l_{k}> m_{k}> n_{k}\geq k$ that $G(Kx_{3l_{k}},Kx_{3m_{k}},Kx_{3n_{k}})\geq \varepsilon_{0}$. Choose $0< \gamma < 1$
such that $\frac{\varepsilon_{0}}{\gamma}> G(A,B,C)$, and choose $\varepsilon >0$ such that
\begin{equation*}
0< \varepsilon < \min\left \{\frac{\varepsilon_{0}}{\gamma}-G(A,B,C), ~\frac{G(A,B,C)\alpha(\gamma)}{1-\alpha(\gamma)}\right\}.
\end{equation*}
Let $N\in \Bbb{N}$ be chosen in such a way that
\begin{align*}
G(Kx_{3n_{k}},Kx_{3n_{k}+1},Kx_{3n_{k}+2})&\leq G(A,B,C)+\varepsilon, ~~\forall n_{k}\geq N, \\
G(Kx_{3m_{k}},Kx_{3n_{k}+1},Kx_{3n_{k}+2})\leq &G(A,B,C)+\varepsilon, ~~\forall m_{k}> n_{k}\geq N, \\
G(Kx_{3l_{k}},Kx_{3n_{k}+1},Kx_{3n_{k}+2})\leq G&(A,B,C)+\varepsilon, ~~\forall l_{k}> m_{k}> n_{k}\geq N.
\end{align*}
Uniform convexity of $(X,G)$ implies that
\begin{flalign*}
G(A,B,C)&\leq G(Kx_{3n_{k}+1},Kx_{3n_{k}+2},\Bbb{W}(Kx_{3n_{k}},Kx_{3m_{k}},Kx_{3l_{k}},\frac{1}{3},\frac{1}{3},
\frac{1}{3})) \\
&\leq (G(A,B,C)+\varepsilon)(1-\alpha(\gamma))< G(A,B,C),
\end{flalign*}
which is a contradiction. Therefore, $(Kx_{3n})$ is a $G$-Cauchy sequence in $A$. By the fact that $K$ is tripartite relatively
anti-Lipschitzian on $A\cup B\cup C$, we have
\begin{equation*}
G(x_{3l},x_{3m},x_{3n})\leq c\, G(Kx_{3l},Kx_{3m},Kx_{3n})\rightarrow 0,~~l,m,n\rightarrow \infty,
\end{equation*}
that is, $(x_{3n})$ is $G$-Cauchy. Since $A$ is $G$-complete, there exists $p\in A$ such that $x_{3n}\rightarrow p$. Now, the
result follows from a similar argument used in the proof of Theorem \ref{6}.
\end{proof}
\section{Tripartite semi-contractions}
In this section we introduce tripartite semi-contractions and establish results on the existence and
convergence of tripartite coincidence-best proximity points for this mappings.
\begin{definition}
Let $A,B$ and $C$ be nonempty subsets of a $G$-metric space $(X,G)$ and $T,S,K:A\cup B\cup C\rightarrow A\cup B\cup C$ be
three mappings. The triple $(T;S;K)$ is called a \textsl{tripartite semi-contraction} if:

(1) $(T;S;K)$ is an \textbf{RLN} on $A\cup B\cup C$.

(2) For some $r\in (0,1)$ and for all $(x,y,z)\in A\times B\times C$ we have
\begin{equation*}
G(Tx,Sy,Kz)\leq rG(Kx,Ky,Kz)+(1-r)G(A,B,C)
\end{equation*}
and
\begin{equation*}
G(Sx,Sy,Sz)\leq rG(Tx,Sy,Kz).
\end{equation*}
\end{definition}
\begin{example}
Let $X:=\Bbb{R}$ and let
\begin{equation*}
G(x,y,z)=|x-y|+|y-z|+|z-x|
\end{equation*}
be a $G$-metric on $X$. In addition, let
\begin{equation*}
A=[0,1], ~~B=[0,2], ~~C=[0,3].
\end{equation*}
Then for mappings $T,S,K:A\cup B\cup C\rightarrow A\cup B\cup C$ defined by
\begin{equation*}
Tx:=0, ~~Sx:=0, ~~Kx:=
\begin{cases}
x ~~if~ x\in A\cup B \\
0 ~~if~ x\in C
\end{cases}
\end{equation*}
for each $(x,y,z)\in A\times B\times C$, we have
\begin{flalign*}
G(Sx,Sy,Sz)=0\leq rG(Tx,Sy,Kz),
\end{flalign*}
and
\begin{flalign*}
G(Tx,Sy,Kz)=0 &\leq rG(Kx,Ky,Kz) \\
&\leq rG(Kx,Ky,Kz)+(1-r)G(A,B,C).
\end{flalign*}
This implies that $(T;S;K)$ is a tripartite semi-contraction for each $r\in (0,1)$.
\end{example}
\begin{remark}
Notice that the condition (2) of the above definition implies that
$$G(Sx,Sy,Sz)\leq G(Kx,Ky,Kz),\quad \forall (x,y,z)\in A\times B\times C.$$
Moreover, if $K$ is a noncyclic, relatively nonexpansive mapping; meaning that
$$G(Kx,Ky,Kz)\leq G(x,y,z)\quad \forall (x,y,z)\in A\times B\times C,$$
then $S$ is a left cyclic contraction. In addition, if in the above definition $K$ is $G$-continuous, then $S$ is $G$-continuous as well (that is,
they are continuous with respect to the topology induced by the $G$-metric).
\end{remark}
\begin{lemma}\label{104}
Let $(A,B,C)$ be a triple of nonempty subsets of a $G$-metric space $(X,G)$ and let $(T;S;K)$ be a tripartite semi-contraction on
$A\cup B\cup C$. Suppose that $T(A)\subseteq S(C)\subseteq K(B)$, $T(B)\subseteq S(A)\subseteq K(C)$ and
$T(C)\subseteq S(B)\subseteq K(A)$. Then there exists a sequence $(x_{n})$ in $X$ such that $Tx_{n}=Sx_{n+1}=Kx_{n+2}$
for each $n\geq 0$; moreover, $(x_{3n}), ~(x_{3n+1})$ and $(x_{3n+2})$ are sequences in $A,C$ and $B$ respectively, and
\begin{equation*}
G(Tx_{3n},Tx_{3n+1},Tx_{3n+2})\rightarrow 0,\quad n\to\infty.
\end{equation*}
\end{lemma}
\begin{proof}
Let $x_{0}\in A$. Since $Tx_{0}\in S(C)\subseteq K(B)$ and $T(C)\subseteq S(B)$, there exists $x_{1}\in C, ~x_{2}\in B$ such that
$Tx_{0}=Sx_{1}=Kx_{2}$ and $Tx_{1}=Sx_{2}\in K(A)$. We know that $T(B)\subseteq S(A)$, therefore there exists $x_{3}\in A$
such that $Tx_{1}=Sx_{2}=Kx_{3}$ and $Tx_{2}=Sx_{3}\in K(C)$.

Again, since $T(A)\subseteq S(C)$, there exists $x_{4}\in C$ such that $Tx_{2}=Sx_{3}=Kx_{4}$ and $Tx_{3}=Sx_{4}\in K(B)$.
Since $T(C)\subseteq S(B)$, there exists $x_{5}\in B$ such that $Tx_{3}=Sx_{4}=Kx_{5}$ and $Tx_{4}=Sx_{5}\in K(A)$.

Continuing this process, we obtain a sequence $(x_{n})$ such that
$(x_{3n}), ~(x_{3n+1})$ and $(x_{3n+2})$ are sequences in $A,C$ and $B$ respectively and $Tx_{n}=Sx_{n+1}=Kx_{n+2}$
for each $n\geq 0$. Since $(T;S;K)$ is a tripartite semi-contraction we have
\begin{flalign*}
G(Tx_{3n},Tx_{3n+1},Tx_{3n+2})&=G(Sx_{3n+1},Sx_{3n+2},Sx_{3n+3}) \\
&\leq rG(Tx_{3n+1},Sx_{3n+2},Kx_{3n+3})\leq rG(Kx_{3n+1},Kx_{3n+2},Kx_{3n+3}) \\
&=rG(Sx_{3n},Sx_{3n+1},Sx_{3n+2}) \\
&\leq ... \leq r^{3n}G(Sx_{1},Sx_{2},Sx_{3})=r^{3n}G(Tx_{0},Tx_{1},Tx_{2}).
\end{flalign*}

Letting $n\rightarrow \infty$, we have
\begin{equation*}
G(Tx_{3n},Tx_{3n+1},Tx_{3n+2})\rightarrow 0.
\end{equation*}
\end{proof}

Note that in the above lemma, for each $n\geq 1$, $(x_{3n-1})\subseteq B$ and for each $n\geq 2$, $(x_{3n-2})\subseteq C$.
\begin{lemma}\label{105}
Let $(A,B,C)$ be a nonempty triple of subsets of a $G$-metric space $(X,G)$ and let $(T;S;K)$ be a tripartite semi-contraction on
$A\cup B\cup C$. Suppose that $T(A)\subseteq S(C)\subseteq K(B)$, $T(B)\subseteq S(A)\subseteq K(C)$ and
$T(C)\subseteq S(B)\subseteq K(A)$. For $x_{0}\in A$, define $Tx_{n}=Sx_{n+1}=Kx_{n+2}$
for each $n\geq 0$. Then we have
\begin{equation*}
G(Kx_{3n},Kx_{3n+1},Kx_{3n+2})\rightarrow G(A,B,C).
\end{equation*}
\end{lemma}
\begin{proof}
Since $(T;S;K)$ is tripartite semi-contraction we have
\begin{flalign*}
G(Kx_{3n},Kx_{3n+1},Kx_{3n+2})&=G(Sx_{3n-1},Sx_{3n},Sx_{3n+1}) \\
&\leq G(Tx_{3n-1},Sx_{3n},Kx_{3n+1}) \\
&\leq rG(Kx_{3n-1},Kx_{3n},Kx_{3n+1})+(1-r)G(A,B,C) \\
&=rG(Sx_{3n-2},Sx_{3n-1},Sx_{3n})+(1-r)G(A,B,C) \\
&\leq rG(Tx_{3n-2},Sx_{3n-1},Kx_{3n})+(1-r)G(A,B,C) \\
&\leq r[rG(Kx_{3n-2},Kx_{3n-1},Kx_{3n})+(1-r)G(A,B,C)]+(1-r)G(A,B,C) \\
&=r^{2}G(Kx_{3n-2},Kx_{3n-1},Kx_{3n})+(1-r^{2})G(A,B,C) \\
&\leq ... \leq r^{3n}G(Kx_{0},Kx_{1},Kx_{2})+(1-r^{3n})G(A,B,C).
\end{flalign*}

Letting $n\rightarrow \infty$, we obtain
\begin{equation*}
G(Kx_{3n},Kx_{3n+1},Kx_{3n+2})\rightarrow G(A,B,C).
\end{equation*}
\end{proof}

In the following, we shall establish a theorem on the existence of tripartite coincidence-best proximity point.
\begin{theorem}\label{106}
Let $(A,B,C)$ be a nonempty triple of subsets of a $G$-metric space $(X,G)$ and let $(T;S;K)$ be a tripartite semi-contraction on
$A\cup B\cup C$. Suppose that $T(A)\subseteq S(C)\subseteq K(B)$, $T(B)\subseteq S(A)\subseteq K(C)$ and
$T(C)\subseteq S(B)\subseteq K(A)$ and $S$ is $G$-continuous on $A\cup B\cup C$. For $x_{0}\in A$, define $Tx_{n}=Sx_{n+1}=Kx_{n+2}$ for each $n\geq 0$.
If $(x_{3n})$ has a $G$-convergent subsequence in $A$, then the triple $(T;S;K)$ has a tripartite coincidence-best proximity
point in $A$.
\end{theorem}
\begin{proof}
Let $(x_{3n_{k}})$ be a subsequence of $(x_{3n})$ such that $x_{3n_{k}}\rightarrow p\in A$ (in the topology induced by the $G$-metric).
By Lemmas \ref{104} and \ref{105}, if
$k\rightarrow \infty$, we obtain
\begin{equation*}
G(Tp,Tx_{3n_{k}-1},Tx_{3n_{k}-2})\rightarrow 0,
\end{equation*}
and
\begin{equation*}
G(Kp,Kx_{3n_{k}+1},Kx_{3n_{k}+2})\rightarrow G(A,B,C).
\end{equation*}
It follows from the continuity of $S$ and Theorem \ref{2} that
\begin{align*}
G(A,B,C)&\leq G(Kp,Tp,Sp)\leq G(Kp,Sx_{3n_{k}},Sx_{3n_{k}})+G(Tp,Sx_{3n_{k}},Sx_{3n_{k}})+G(Sp,Sx_{3n_{k}},Sx_{3n_{k}}) \\
&\leq G(Kp,Kx_{3n_{k}+1},Kx_{3n_{k}+1})+G(Tp,Tx_{3n_{k}-1},Tx_{3n_{k}-1})+G(Sp,Sx_{3n_{k}},Sx_{3n_{k}}) \\
&\leq G(Kp,Kx_{3n_{k}+1},Kx_{3n_{k}+2})+G(Tp,Tx_{3n_{k}-1},Tx_{3n_{k}-2})+G(Sp,Sx_{3n_{k}},Sx_{3n_{k}}).
\end{align*}
Letting $k\rightarrow \infty$, we conclude that
\begin{equation*}
G(Kp,Tp,Sp)=G(A,B,C).
\end{equation*}
\end{proof}
\begin{lemma}\label{107}
Let $(A,B,C)$ be a nonempty triple of subsets of a $G$-metric space $(X,G)$ and let $(T;S;K)$ be a tripartite semi-contraction on
$A\cup B\cup C$. Suppose that $T(A)\subseteq S(C)\subseteq K(B)$, $T(B)\subseteq S(A)\subseteq K(C)$ and
$T(C)\subseteq S(B)\subseteq K(A)$, moreover $S$ and $K$ commute on $A\cup C$. For $x_{0}\in A$, define $Tx_{n}=Sx_{n+1}=Kx_{n+2}$
for each $n\geq 0$. Then $(Kx_{3n})$, $(Kx_{3n+1})$ and $(Kx_{3n+2})$ are $G$-bounded sequences in $A,\, C$ and $B$, respectively.
\end{lemma}
\begin{proof}
The proof is similar to that of Lemma \ref{7}.
\end{proof}

The next theorem is a straightforward consequence of Theorem \ref{106} and Lemma \ref{107}.
\begin{theorem}\label{1011}
Let $(A,B,C)$ be a nonempty triple of subsets of a $G$-metric space $(X,G)$ such that $A$ is $G$-sequentially compact, and let
$(T;S;K)$ be a tripartite semi-contraction on $A\cup B\cup C$. Suppose that $T(A)\subseteq S(C)\subseteq K(B)$,
$T(B)\subseteq S(A)\subseteq K(C)$ and $T(C)\subseteq S(B)\subseteq K(A)$, moreover $S$ and $K$ commute on $A\cup C$. If $K$ is
tripartite relatively anti-Lipschitzian and $G$-continuous on $A\cup B\cup C$, then the $(T;S;K)$ has a tripartite
coincidence-best proximity point in $A$.
\end{theorem}

We now illustrate Theorem \ref{1011} with the following examples.
\begin{example}
Consider $X:=\Bbb{R}$ with the usual metric. We have already seen that the function $G:X\times X\times X\rightarrow [0,+\infty)$ defined by
\begin{equation*}
G(x,y,z)=\max\{|x-y|,|y-z|,|z-x|\},
\end{equation*}
for all $x,y,z\in X$, is a $G$-metric on $X$. For $A=C=[0,1], B=[-1,0]$ we have $G(A,B,C)=0$. Define
\begin{equation*}
Tx:=0, ~~Sx:=
\begin{cases}
\frac{1}{4}x ~~if~ x\in C \\
0 ~~~~if~ x\in A\cup B
\end{cases}
, ~~Kx:=\frac{1}{2}x, ~~\forall x\in A\cup B\cup C.
\end{equation*}
Then for each $x,z\in [0,1]$ and $y\in [-1,0]$, since $x,-y,z\in [0,1]$, we may assume that $x\leq z\leq -y$, so that we have
\begin{flalign*}
G(Sx,Sy,Sz)&=\max\{|Sx-Sy|,|Sy-Sz|,|Sz-Sx|\}=\max\{|0-0|,|0-\frac{1}{4}z|,|\frac{1}{4}z-0|\} \\
&=\frac{1}{4}z=\frac{1}{2}(\frac{1}{2}z)=\frac{1}{2}\max\{|0-0|,|0-\frac{1}{2}z|,|\frac{1}{2}z-0|\} \\
&=\frac{1}{2}G(Tx,Sy,Kz).
\end{flalign*}
In addition, we see that
\begin{flalign*}
G(Tx,Sy,Kz)&=\max\{|0-0|,|0-\frac{1}{2}z|,|\frac{1}{2}z-0|\}=\frac{1}{2}z \\
&=\frac{1}{4}(2z)=\frac{1}{2}(\frac{1}{2}(2z))=\frac{1}{2}(\frac{1}{2}(z+z))\leq \frac{1}{2}|\frac{1}{2}z-\frac{1}{2}y|=
\frac{1}{2}|\frac{1}{2}y-\frac{1}{2}z| \\
&\leq \frac{1}{2}\max\{|\frac{1}{2}x-\frac{1}{2}y|,
|\frac{1}{2}y-\frac{1}{2}z|,|\frac{1}{2}z-\frac{1}{2}x|\} \\
&=\frac{1}{2}G(Kx,Ky,Kz) \\
&\leq \frac{1}{2}G(Kx,Ky,Kz)+\frac{1}{2}G(A,B,C).
\end{flalign*}
This implies that $(T;S;K)$ is a tripartite semi-contraction with $r=\frac{1}{2}$. Also, $T(A)\subseteq S(C)\subseteq K(B)$,
$T(B)\subseteq S(A)\subseteq K(C)$ and $T(C)\subseteq S(B)\subseteq K(A)$. Moreover, $K$ is $G$-continuous on
$A\cup B\cup C$ and $A$ is $G$-sequentially compact in $X$.
Besides, $K$ is tripartite relatively anti-Lipschitzian on $A\cup B\cup C$ with $c=3$. In fact, for all $(x,y,z)\in A\times B\times C$
we have
\begin{flalign*}
G(x,y,z)&=\max\{|x-y|,|y-z|,|z-x|\} \\
&\leq 3\max\{|\frac{1}{2}x-\frac{1}{2}y|,|\frac{1}{2}y-\frac{1}{2}z|,|\frac{1}{2}z-\frac{1}{2}x|\} \\
&=3G(Kx,Ky,Kz).
\end{flalign*}
Finally, for each $x\in A\cup C$, it is clear that
\begin{equation*}
K(Sx)=S(Kx),
\end{equation*}
that is, $S$ and $K$ commute on $A \cup C$. Thereby, the existence of tripartite coincidence-best proximity point for
$(T;S;K)$ follows from Theorem \ref{1011}. That is, there exists $p\in A$ such that
\begin{equation*}
G(Kp,Tp,Sp)=G(A,B,C),
\end{equation*}
or
\begin{equation*}
G(\frac{1}{2}p,0,0)=0.
\end{equation*}
which implies that $p=0$.
\end{example}
\begin{example}
Consider $X:=\Bbb{R}$ with the usual metric. We have already seen that the function $G:X\times X\times X\rightarrow [0,+\infty)$ defined by
\begin{equation*}
G(x,y,z)=\max\{|x-y|,|y-z|,|z-x|\},
\end{equation*}
for all $x,y,z\in X$, is a $G$-metric on $X$. For $A=C=[0,1], B=[-1,0]$ we have $G(A,B,C)=0$. We now consider
\begin{equation*}
Tx:=0, ~~Sx:=
\begin{cases}
\frac{1}{2}\sin x ~~~ x\in C \\
0 ~~~~~~ x\in A\cup B
\end{cases}
, ~~Kx:=x, ~~\forall x\in A\cup B\cup C.
\end{equation*}
Then for each $x,z\in [0,1]$ and $y\in [-1,0]$, since $x,-y,z\in [0,1]$, we may assume that $x\leq z\leq -y$, therefore we have
\begin{flalign*}
G(Sx,Sy,Sz)&=\max\{|Sx-Sy|,|Sy-Sz|,|Sz-Sx|\}=\max\{|0-0|,|0-\frac{1}{2}\sin z|,|\frac{1}{2}\sin z-0|\} \\
&=\frac{1}{2}\sin z\leq \frac{1}{2}z=\frac{1}{2}\max\{|0-0|,|0-z|,|z-0|\} \\
&=\frac{1}{2}G(Tx,Sy,Kz).
\end{flalign*}
In addition, we see that
\begin{flalign*}
G(Tx,Sy,Kz)&=\max\{|0-0|,|0-z|,|z-0|\}=z \\
&=\frac{1}{2}(2z)=\frac{1}{2}(z+z)\leq \frac{1}{2}|z-y|=\frac{1}{2}|y-z| \\
&\leq \frac{1}{2}\max\{|x-y|,|y-z|,|z-x|\} \\
&=\frac{1}{2}G(Kx,Ky,Kz) \\
&\leq \frac{1}{2}G(Kx,Ky,Kz)+\frac{1}{2}G(A,B,C).
\end{flalign*}
This implies that $(T;S;K)$ is a tripartite semi-contraction with $r=\frac{1}{2}$. Also, $T(A)\subseteq S(C)\subseteq K(B)$,
$T(B)\subseteq S(A)\subseteq K(C)$ and $T(C)\subseteq S(B)\subseteq K(A)$. Moreover, $K$ is $G$-continuous on
$A\cup B\cup C$ and $A$ is $G$-sequentially compact in $X$.
Besides, $K$ is tripartite relatively anti-Lipschitzian on $A\cup B\cup C$ with $c=3$. In fact, for all $(x,y,z)\in A\times B\times C$
we have
\begin{flalign*}
G(x,y,z)&=\max\{|x-y|,|y-z|,|z-x|\} \\
&\leq 3\max\{|x-y|,|y-z|,|z-x|\} \\
&=3G(Kx,Ky,Kz).
\end{flalign*}
Finally, for each $x\in A\cup C$ it is clear that
\begin{equation*}
K(Sx)=S(Kx),
\end{equation*}
that is, $S$ and $K$ commute on $A \cup C$. Now, the existence of tripartite coincidence-best proximity point for
$(T;S;K)$ follows from Theorem \ref{1011}. This means that there exists $p\in A$ such that
\begin{equation*}
G(Kp,Tp,Sp)=G(A,B,C),
\end{equation*}
or
\begin{equation*}
G(p,0,0)=0.
\end{equation*}
from which it follows that $p=0$.
\end{example}
So far, we have been dealing with the existence of tripartite coincidence-best proximity points for tripartite semi-contractions.
Now we want to approximate these points. To achieve this goal, we  need the convex structure of $G$-metric space.
\begin{lemma}\label{108}
Let $(A,B,C)$ be a nonempty triple of subsets of a uniformly convex $G$-metric space $(X,G;\Bbb{W})$ such that $A$ is $G$-convex.
Let $(T;S;K)$ be a tripartite semi-contraction on $A\cup B\cup C$ such that $T(A)\subseteq S(C)\subseteq K(B)$,
$T(B)\subseteq S(A)\subseteq K(C)$ and $T(C)\subseteq S(B)\subseteq K(A)$. For $x_{0}\in A$, define $Tx_{n}=Sx_{n+1}=Kx_{n+2}$
for each $n\geq 0$. Then
\begin{align*}
G(Kx_{3n},Kx_{3n+3},Kx_{3n+6})\rightarrow 0,
\end{align*}
\begin{align*}
G(Kx_{3n+1},Kx_{3n+4},Kx_{3n+7})\rightarrow 0,
\end{align*}
and
\begin{equation*}
G(Kx_{3n+2},Kx_{3n+5},Kx_{3n+8})\rightarrow 0.
\end{equation*}
\end{lemma}
\begin{proof}
The proof is essentially the same as that of Lemma \ref{8}. We omit the details.
\end{proof}
\begin{theorem}\label{109}
Let $(A,B,C)$ be a triple of nonempty, closed subsets of a complete uniformly convex $G$-metric space $(X,G,\Bbb{W})$
such that $A$ is $G$-convex. Let $(T;S;K)$ be a tripartite semi-contraction on $A\cup B\cup C$ such that
 $T(A)\subseteq S(C)\subseteq K(B)$, $T(B)\subseteq S(A)\subseteq K(C)$ and $T(C)\subseteq S(B)\subseteq K(A)$ and that
 $K$ is $G$-continuous and tripartite relatively anti-Lipschitzian on $A\cup B\cup C$. Then
 $(T;S;K)$ has a tripartite coincidence-best proximity point in $A$. Moreover, if $x_{0}\in A$ and $Tx_{n}=Sx_{n+1}=Kx_{n+2}$,
 then $(x_{3n})$ $G$-converges to the tripartite coincidence-best proximity point of $(T;S;K)$.
\end{theorem}
\begin{proof}
For $x_{0}\in A$ define $Tx_{n}=Sx_{n+1}=Kx_{n+2}$ for each $n\geq 0$. We prove that $(Kx_{3n})$, $(Kx_{3n+1})$ and $(Kx_{3n+2})$
are $G$-Cauchy sequences. At first, we verify that for each $\varepsilon >0$ there exists $N_{0}\in \Bbb{N}$ such that
\begin{equation*}
G(Kx_{3m},Kx_{3n+1},Kx_{3n+2})< G(A,B,C)+\varepsilon, ~~\forall m>n\geq N_{0}. ~~~~~(*)
\end{equation*}
Assume to the contrary that there exists $\varepsilon_{0}>0$ such that for each $k\geq 1$ there exist $m_{k}>n_{k}\geq k$
satisfying
\begin{align*}
G(Kx_{3m_{k}},Kx_{3n_{k}+1},Kx_{3n_{k}+2})\geq G(A,B,C)+\varepsilon_{0},\end{align*}
and
\begin{align*}
G(Kx_{3m_{k}-3},Kx_{3n_{k}+1},Kx_{3n_{k}+2})< G(A,B,C)+\varepsilon_{0}.
\end{align*}
Now, we have
\begin{flalign*}
G(A,B,C)+\varepsilon_{0}&\leq G(Kx_{3m_{k}},Kx_{3n_{k}+1},Kx_{3n_{k}+2})=G(Sx_{3m_{k}-1},Sx_{3n_{k}},Sx_{3n_{k}+1}) \\
&\leq G(Tx_{3m_{k}-1},Sx_{3n_{k}},Kx_{3n_{k}+1}) \\
&\leq G(Kx_{3m_{k}-1},Kx_{3n_{k}},Kx_{3n_{k}+1})=G(Sx_{3m_{k}-2},Sx_{3n_{k}-1},Sx_{3n_{k}}) \\
&\leq G(Tx_{3m_{k}-2},Sx_{3n_{k}-1},Kx_{3n_{k}}) \\
&\leq G(Kx_{3m_{k}-2},Kx_{3n_{k}-1},Kx_{3n_{k}})=G(Sx_{3m_{k}-3},Sx_{3n_{k}-2},Sx_{3n_{k}-1}) \\
&\leq G(Tx_{3m_{k}-3},Sx_{3n_{k}-2},Kx_{3n_{k}-1}) \\
&\leq G(Kx_{3m_{k}-3},Kx_{3n_{k}-2},Kx_{3n_{k}-1}) \\
&\leq G(Kx_{3m_{k}-3},Kx_{3n_{k}+1},Kx_{3n_{k}+1})+G(Kx_{3n_{k}-2},Kx_{3n_{k}+1},Kx_{3n_{k}+1}) \\
&+G(Kx_{3n_{k}-1},Kx_{3n_{k}+1},Kx_{3n_{k}+1}) \\
&\leq G(Kx_{3m_{k}-3},Kx_{3n_{k}+1},Kx_{3n_{k}+2})+G(Kx_{3n_{k}-2},Kx_{3n_{k}+1},Kx_{3n_{k}+4}) \\
&+ G(Kx_{3n_{k}-1},Kx_{3n_{k}},Kx_{3n_{k}+1}) \\
&= G(Kx_{3m_{k}-3},Kx_{3n_{k}+1},Kx_{3n_{k}+2})+G(Kx_{3n_{k}-2},Kx_{3n_{k}+1},Kx_{3n_{k}+4}) \\
&+ G(Tx_{3n_{k}-3},Tx_{3n_{k}-2},Tx_{3n_{k}-1}).
\end{flalign*}
Letting $k\rightarrow \infty$, and using the hypothesis together with Lemmas \ref{104} and \ref{108} we obtain
\begin{equation*}
G(Kx_{3m_{k}},Kx_{3n_{k}+1},Kx_{3n_{k}+2})\rightarrow G(A,B,C)+\varepsilon_{0}.
\end{equation*}
Besides,
\begin{flalign*}
G(A,B,C)+\varepsilon_{0}&\leq G(Kx_{3m_{k}},Kx_{3n_{k}+1},Kx_{3n_{k}+2})=G(Sx_{3m_{k}-1},Sx_{3n_{k}},Sx_{3n_{k}+1}) \\
&\leq G(Tx_{3m_{k}-1},Sx_{3n_{k}},Kx_{3n_{k}+1}) \\
&\leq G(Kx_{3m_{k}-1},Kx_{3n_{k}},Kx_{3n_{k}+1})=G(Sx_{3m_{k}-2},Sx_{3n_{k}-1},Sx_{3n_{k}}) \\
&\leq G(Tx_{3m_{k}-2},Sx_{3n_{k}-1},Kx_{3n_{k}}) \\
&\leq G(Kx_{3m_{k}-2},Kx_{3n_{k}-1},Kx_{3n_{k}})=G(Sx_{3m_{k}-3},Sx_{3n_{k}-2},Sx_{3n_{k}-1}) \\
&\leq G(Tx_{3m_{k}-3},Sx_{3n_{k}-2},Kx_{3n_{k}-1}) \\
&\leq rG(Kx_{3m_{k}-3},Kx_{3n_{k}-2},Kx_{3n_{k}-1})+(1-r)G(A,B,C).
\end{flalign*}
Letting $k\rightarrow \infty$, we conclude that
\begin{equation*}
G(A,B,C)+\varepsilon_{0}\leq r(G(A,B,C)+\varepsilon_{0})+(1-r)G(A,B,C)\leq G(A,B,C)+\varepsilon_{0}.
\end{equation*}
This implies that $r=1$, which is a contradiction. That is, $(*)$ holds. Similarly, we see that
\begin{equation*}
G(Kx_{3l_{k}},Kx_{3l_{k}+1},Kx_{3l_{k}+2})< G(A,B,C)+\varepsilon_{0}.
\end{equation*}
Now, suppose $(Kx_{3n})$ is not a $G$-Cauchy sequence. Then there exists $\varepsilon_{0}> 0$ such that for each $k\geq 1$ there exist
$l_{k}> m_{k}> n_{k}\geq k$ in such a way that $G(Kx_{3l_{k}},Kx_{3m_{k}},Kx_{3n_{k}})\geq \varepsilon_{0}$. Choose $0< \gamma < 1$
such that $\frac{\varepsilon_{0}}{\gamma}> G(A,B,C)$ and choose $\varepsilon >0$ such that
\begin{equation*}
0< \varepsilon < \min\left \{\frac{\varepsilon_{0}}{\gamma}-G(A,B,C), ~\frac{G(A,B,C)\alpha(\gamma)}{1-\alpha(\gamma)}\right\}.
\end{equation*}
Let $N\in \Bbb{N}$ be such that
\begin{align*}
G(Kx_{3n_{k}},Kx_{3n_{k}+1},Kx_{3n_{k}+2})&\leq G(A,B,C)+\varepsilon, ~~\forall n_{k}\geq N, \\
G(Kx_{3m_{k}},Kx_{3n_{k}+1},Kx_{3n_{k}+2})\leq &G(A,B,C)+\varepsilon, ~~\forall m_{k}> n_{k}\geq N, \\
G(Kx_{3l_{k}},Kx_{3n_{k}+1},Kx_{3n_{k}+2})\leq G&(A,B,C)+\varepsilon, ~~\forall l_{k}> m_{k}> n_{k}\geq N.
\end{align*}
Uniform convexity of $(X,G)$ now implies that
\begin{flalign*}
G(A,B,C)&\leq G(Kx_{3n_{k}+1},Kx_{3n_{k}+2},\Bbb{W}(Kx_{3n_{k}},Kx_{3m_{k}},Kx_{3l_{k}},\frac{1}{3},\frac{1}{3},
\frac{1}{3})) \\
&\leq (G(A,B,C)+\varepsilon)(1-\alpha(\gamma))< G(A,B,C),
\end{flalign*}
which is a contradiction. Therefore, $(Kx_{3n})$ is a $G$-Cauchy sequence in $A$. By the fact that $K$ is tripartite relatively
anti-Lipschitzian on $A\cup B\cup C$, we have
\begin{equation*}
G(x_{3l},x_{3m},x_{3n})\leq c\, G(Kx_{3l},Kx_{3m},Kx_{3n})\rightarrow 0,~~l,m,n\rightarrow \infty,
\end{equation*}
that is, $(x_{3n})$ is $G$-Cauchy. Since $A$ is $G$-complete, there exists $p\in A$ such that $x_{3n}\rightarrow p$. Now, the
result follows from a similar argument used in the proof of Theorem \ref{106}.
\end{proof}
\bibliographystyle{amsplain}

\end{document}